\documentclass[a4paper,12pt,reqno]{amsart}
\usepackage{amsmath, epsfig, cite}
\usepackage{cases}
\usepackage{amssymb}
\usepackage{amsfonts}
\usepackage{latexsym}
\usepackage{amsthm}
\usepackage{url}
\usepackage{color}
\usepackage{comment}
\usepackage{a4wide}
\usepackage{enumerate}
\usepackage{bbm}

\newtheorem{lemma}{Lemma}[section]

\newtheorem{thm}{Theorem}[section]

\theoremstyle{definition}

\numberwithin{equation}{section}

\allowdisplaybreaks[1]

\begin{document}
\title{On pairs of one prime,four prime cubes \\and powers of 2}
\author{XIN CHEN}

 \newcommand{\mods}[1]{\,(\mathrm{mod}\,{#1})}
\begin{abstract}
In this paper, we consider the simultaneous representation of pairs of sufficiently large integers. We prove that every pair of large positive odd integers can be represented in the form of a pair of one prime, four cubes of primes and 231 powers of 2.
\end{abstract}
\dedicatory{School of Mathematical Science,\\ Shanghai Jiao Tong University\\Shanghai 200240,P.R.China\\chen3jin@sjtu.edu.cn}

\maketitle

\section{\bf Introduction }

The Goldbach conjecture is one of the most famous problems and there are many variations derived from the conjecture.
In 1951, Linnik \cite{1951l} proved under the assumption of the Generalized Riemann Hypothesis (GRH) that every large even integer $N$ can be written as the sum of two primes and finite number of powers of 2, and later in 1953 \cite{1953l} he proved this conjecture unconditionally: that is
\begin{equation}
N = p_{1}+p_{2}+2^{v_{1}}+2^{v_{2}}+\cdots+2^{v_{k}}.
\end{equation}
The explicit value of the number $k$ was first obtained by Liu, Liu and Wang \cite{1998w}, in which $k=54000$ is acceptable. Afterwards, several researchers improved the value of $k$ and in 2002 Heath-Brown and Puchta \cite{2002h} showed that $k=13$ and ,under the GRH, $k=7$.

In 1999, Liu,Liu and Zhan \cite{1999l} proved every large even integer $N$ can be written as a sum of four squares of primes and $k_{1}$ powers of 2,
\begin{equation}
N = p_{1}^{2}+p_{2}^{2}+p_{3}^{2}+p_{4}^{2}+2^{v_{1}}+2^{v_{2}}+\cdots+2^{v_{k_{1}}}.
\end{equation}
In 2001, Liu and Liu \cite{2001l} proved that every large even integer $N$ can be written as a sum of eight cubes of primes and $k_{2}$ powers of 2,
\begin{equation}
N = p_{1}^{3}+p_{2}^{3}+\cdots+p_{8}^{3}+2^{v_{1}}+2^{v_{2}}+\cdots+2^{v_{k_{2}}}.
\end{equation}
Other problems with hybrid powers of primes of (1.1-3) have been studied by Liu, Liu and Zhan \cite{1999l}, Li \cite{2007l}, Liu and L{\"u} \cite{2011l1} and Liu and L{\"u} \cite{2011l2}. The detail is omitted.

In 2013, Kong \cite{2013k} studied a simultaneous version of the Goldbach-Linnik problem. Kong \cite{2013k} proved that the simultaneous equations
\begin{equation} \label{1}
\begin{cases}
   N_{1} = p_{1}+p_{2}+2^{v_{1}}+2^{v_{2}}+\cdots+2^{v_{k}}\\
   N_{2} = p_{3}+p_{4}+2^{v_{1}}+2^{v_{2}}+\cdots+2^{v_{k}}
\end{cases}
\end{equation}
are solvable for every pair of sufficiently large positive even integers $N_{1}$, $N_{2}$ satisfying $N_{2}\gg N_{1}\geq N_{2}$ for $k=63$ in general and for $k=31$ under the GRH. Then the result was improved by Kong \cite{2017k} in 2017, which showed that the simultaneous equations (\ref{1}) can be solvable for $k=34$ unconditionally and $k=18$ under the GRH. In 2013, Liu \cite{2013l} first considered the result on simultaneous representation of pairs of positive odd integers $N_{2}\gg N_{1}\geq N_{2}$, in the form
\begin{equation} \label{2}
\begin{cases}
   N_{1} = p_{1}+p_{2}^{2}+p_{3}^{2}+2^{v_{1}}+2^{v_{2}}+\cdots+2^{v_{k_{1}}}\\
   N_{2} = p_{4}+p_{5}^{2}+p_{6}^{2}+2^{v_{1}}+2^{v_{2}}+\cdots+2^{v_{k_{1}}}
\end{cases}
\end{equation}
He proved that the equations (\ref{2}) are solvable for $k=332$, and then the consequence was advanced by Hu and Yang \cite{2016h} for $k_{2}=128$. In 2017, Hu and Yang \cite{2017h} proved every pair of large integers with $N_{2}\gg N_{1}\geq N_{2}$ can be represented in the form
\begin{equation} \label{3}
\begin{cases}
   N_{1} = p_{1}+p_{2}^{2}+p_{3}^{3}+p_{4}^{3}+2^{v_{1}}+2^{v_{2}}+\cdots+2^{v_{k_{2}}}\\
   N_{2} = p_{5}+p_{6}^{2}+p_{7}^{3}+p_{8}^{3}+2^{v_{1}}+2^{v_{2}}+\cdots+2^{v_{k_{2}}}
\end{cases}
\end{equation}
for $k_{2}=455$. Cai and Hu \cite{2020c} improved $k_{2}$ to 187 in 2020.

In this paper, we consider a simultaneous representation of the problem studied by Liu and L{\"u} in \cite{2011l2}. We will show that each pair of large odd integers with $N_{2}\gg N_{1}\geq N_{2}$ can be written in the form
\begin{equation} \label{4}
\begin{cases}
   N_{1} = p_{1}+p_{2}^{3}+p_{3}^{3}+p_{4}^{3}+p_{5}^{3}+2^{v_{1}}+2^{v_{2}}+\cdots+2^{v_{k}}\\
   N_{2} = p_{6}+p_{7}^{3}+p_{8}^{3}+p_{9}^{3}+p_{10}^{3}+2^{v_{1}}+2^{v_{2}}+\cdots+2^{v_{k}}.
\end{cases}
\end{equation}
where $p_{i}$ is prime and $k$ is a positive integer. My result is stated as follows.

\begin{thm}
For $k=231$, the simultaneous equations (\ref{4}) are solvable for every pair of sufficiently large positive odd  integers $N_{2}\gg N_{1}\geq N_{2}$.
\end{thm}

\section{\bf Outline of the method }
In this section, we give an outline for the proof of Theorem 1.1.

In order to apply the circle method, we set
\begin{equation}
P_{i}=N_{i}^{1/9-2\varepsilon}, \qquad  Q_{i}=N_{i}^{8/9+\varepsilon}\notag
\end{equation}
for $i=1,2$. By Dirichlet's lemma in \cite{1997v}, each $\alpha\in [1/Q,1+1/Q]$ can be written in the form
\begin{equation}
\alpha = a/q+\theta,\qquad |\theta|\leq1/qQ\notag
\end{equation}
for some integers $a$, $q$ with $1\leq a\leq q\leq Q$ and $(a,q)=1$. According to the lemma, for any integers $a_{i}$, $q_{i}$ satisfying
\begin{equation}
1\leq a_{i}\leq q_{i}\leq P_{i},\ \ (a_{i}, q_{i})=1,\notag
\end{equation}
we define the major arcs $\mathfrak{M}_{i}$ and minor ares $\mathfrak{m}_{i}$ as usual, namely
\begin{equation}
\mathfrak{M}_{i}=\bigcup\limits_{q_{i}\leq P_{i}}\bigcup\limits_{\substack{1\leq a_{i}\leq q_{i} \\ (a_{i},q_{i})=1}}\mathfrak{M}_{i}(a_{i},q_{i}),\quad \mathfrak{m}_{i}=\left [\frac{1}{Q_{i}},1+\frac{1}{Q_{i}}\right ]\backslash\mathfrak{M}_{i},\notag
\end{equation}
where $i=1,2$ and
\begin{equation}
\mathfrak{M}_{i}(a_{i},q_{i})=\left \{\alpha_{i}:\left |\alpha_{i}-\frac{a_{i}}{q_{i}}\right |\leq \frac{1}{q_{i}Q_{i}} \right \}\notag
\end{equation}
It follows from $2P_{i}\leq Q_{i}$ that every pair of the arcs $\mathfrak{M}_{i}$ is mutually disjoint respectively. We further define
\begin{equation}
\mathfrak{M}=\mathfrak{M}_{1}\times\mathfrak{M}_{2}=\{(\alpha_{1},\alpha_{2}):\alpha_{1}\in \mathfrak{M}_{1},\alpha_{2}\in \mathfrak{M}_{2}\},\ \mathfrak{m}=\left [\frac{1}{Q_{i}},1+\frac{1}{Q_{i}}\right ]^2\setminus\mathfrak{M}.\notag
\end{equation}

As the value in \cite{2011l2}, let $\delta=10^{-4}$, and
\begin{equation}\label{2.1}
U_{i}=\left (\frac{N_{i}}{16(1+\delta)} \right )^{1/3},\quad V_{i}=U_{i}^{5/6}, \quad L=\log_{2}\left (\frac{N_{1}}{\log N_{1}}\right )
\end{equation}
for $i=1,2$. Let $\omega$ be a small constant, we set
\begin{equation}\label{10}
\begin{aligned}
& f_{i}(\alpha)=\sum\limits_{\omega N_{i}<p\leq N_{i}}(\log p)e(p\alpha),\quad S_{i}(\alpha)=\sum\limits_{p\sim U_{i}}(\log p)e(p^{3}\alpha), \\ & T_{i}(\alpha)=\sum\limits_{p\sim V_{i}}(\log p)e(p^{3}\alpha),\ G(\alpha)=\sum\limits_{v\leq L}e(2^{v}\alpha),\ \sigma_{\lambda}(\alpha)=\{\alpha\in[0,1]:\ |G(\alpha)|\geq\lambda L\}\\ &
\sigma_{\lambda}(\alpha_{1},\alpha_{2})=\{(\alpha_{1},\alpha_{2})\in[0,1]^{2}:\ |G(\alpha_{1}+\alpha_{2})|\geq\lambda L\},
\end{aligned}
\end{equation}
where $i=1,2$, $e(x):=\exp(2\pi ix)$.

Let
\begin{equation}
R(N_{1},N_{2})=\sum \log p_{1}\cdots\log p_{10}\notag
\end{equation}
be the weighted number of solutions of (\ref{4}) in $(p_{1}, \cdots ,p_{10}, v_{1},\cdots ,v_{k})$ with
\begin{equation}
\begin{aligned}
&\omega N_{1}<p_{1}\leq N_{1},\ p_{2},p_{3}\sim U_{1},\ p_{4},p_{5}\sim V_{1},\\
&\omega N_{2}<p_{6}\leq N_{2},\ p_{7},p_{8}\sim U_{2},\ p_{9},p_{10}\sim V_{2},\ v_{j}\leq L\notag
\end{aligned}
\end{equation}
for $j=1,2,\cdots,k$. Then $R(N_{1},N_{2})$ can be written as
\begin{equation}
\begin{aligned}
R(N_{1},N_{2}) &
=\int_{0}^{1}\int_{0}^{1}\left(\prod_{i=1}^{2}f_{i}(\alpha_{i})S_{i}^{2}(\alpha_{i})T_{i}^{2}(\alpha_{i})e(-\alpha_{i}N_{i})\right)G^{k}(\alpha_{1}+\alpha_{2})d\alpha_{1}d\alpha_{2}\\
&=\left \{ \iint\limits_{\mathfrak{M}} + \iint\limits_{\mathfrak{m}\bigcap\sigma_{\lambda}(\alpha_{1}+\alpha_{2})}+\iint\limits_{\mathfrak{m}\backslash\sigma_{\lambda}(\alpha_{1}+\alpha_{2})}\right \}\cdots d\alpha_{1}d\alpha_{2}\\
& = R_{1}(N_{1},N_{2})+R_{2}(N_{1},N_{2})+R_{3}(N_{1},N_{2}).\notag
\end{aligned}
\end{equation}

We will prove Theorem 1.1 by estimating the term $R_{1}(N_{1},N_{2})$, $R_{2}(N_{1},N_{2})$ and $R_{3}(N_{1},N_{2})$. For $\chi\mod q$, define
\begin{equation}
\begin{aligned}
&C_{1}(\chi, a)=\sum\limits_{h=1}^{q}\overline{\chi}(h)e\left(\frac{ah}{q}\right),\qquad C_{1}(q,a)=C_{1}(\chi^{0},a),\\
&C_{3}(\chi, a)=\sum\limits_{h=1}^{q}\overline{\chi}(h)e\left(\frac{ah^{3}}{q}\right),\qquad C_{3}(q,a)=C_{3}(\chi^{0},a),\notag
\end{aligned}
\end{equation}
where $C_{1}(q,a)$ is the Ramanujan sum and $C_{1}(q,a)=\mu(q)$, if $(a,q)=1$. If $\chi_{1},\ldots,\chi_{5}$ are characters$\mod q$, then we write
\begin{equation}
\begin{aligned}
B(n,q;\chi_{1},\ldots,\chi_{5})=\sum\limits_{a=1 \atop (a,q)=1}^{q}C_{1}(\chi_{1},a)C_{3}(\chi_{2},a)\ldots C_{3}(\chi_{5},a)e\left(-\frac{an}{q}\right)\notag
\end{aligned}
\end{equation}
and
\begin{equation}\label{5}
\begin{aligned}
B(n,q)=B(n,q;\chi_{0},\ldots,\chi_{0}),\qquad A(n,q)=\frac{B(n,q)}{\varphi^{5}(q)},\qquad \mathfrak{S}(n)=\sum\limits_{q=1}^{\infty}A(n,q).
\end{aligned}
\end{equation}

We define
\begin{equation}\label{Xi}
\Xi(N,k)=\{(1-\eta)N\leq n\leq N:n=N-2^{v_{1}}-\cdots-2^{v_{k}}\}.
\end{equation}
with $k\geq 2$ and $N\equiv 1$(mod 2).

\begin{lemma}\label{6}
For $n\in \Xi(N_{i},k)$, we have
\begin{equation}
\int_{\mathfrak{M}_{i}}f_{i}(\alpha)S_{i}^{2}(\alpha)T_{i}^{2}(\alpha)e(-n\alpha)d\alpha=\frac{1}{3^{4}}\mathfrak{S}(n)\mathfrak{J}_{i}(n)+O(N_{i}^{\frac{11}{9}}L^{-1})\notag
\end{equation}
for $i=1,2$, where $\mathfrak{S}(n)$ is defined in (\ref{5}) and
\begin{equation}
\mathfrak{J}_{i}(n)=\sum\limits_{m_{1}+\cdots+m_{5}=n\atop U_{i}^{3}<m_{2},m_{3}\leq 8U_{i}^{3},V_{i}^{3}<m_{4},m_{5}\leq 8V_{i}^{3}}(m_{2}\ldots m_{5})^{-2/3}\notag
\end{equation}
with
\begin{equation}
\mathfrak{S}(n)\geq 0.8842495063,\qquad \mathfrak{J}_{i}(n)\geq 2.7335671N_{i}^{11/9}.\notag
\end{equation}
\end{lemma}

\begin{proof}
The proof of Lemma \ref{6} can be found in \cite{2011l2}.
\end{proof}

\begin{lemma}\label{7}
Let $\Xi(N,k)$ be defined in (\ref{Xi}) with $2\leq k\leq L$ and $N\equiv 1$(mod 2). Then we have
\begin{equation}
\sum\limits_{\substack{n_{1}\in \Xi(N_{1},k)\\ n_{2}\in \Xi(N_{2},k)\\ n_{1}\equiv n_{2}\equiv 1(mod\ 2)}} 1\geq (1-\varepsilon)L^{k}.\notag
\end{equation}
\end{lemma}

\begin{proof}
The proof of Lemma \ref{7} is straightforward, so we omit the detail.
\end{proof}

Now we give several lemmas to estimate the value in $R_{2}(N_{1},N_{2})$ and $R_{3}(N_{1},N_{2})$.

\begin{lemma}\label{8}
(Vinogradov).Let
\begin{equation}
S_{1}(N,\alpha)=\sum\limits_{p\sim N}(\log p)e(p\alpha)\notag
\end{equation}
where $\alpha=a/q+\lambda$ subject to $1\leq a\leq q\leq N$, $(a,q)=1$ and $|\lambda|\leq 1/q^{2}$. Then
\begin{equation}
S_{1}(N,\alpha)\ll (Nq^{-1/2}+N^{4/5}+N^{1/2}q^{1/2})L^{c}.\notag
\end{equation}
\end{lemma}

\begin{lemma}\label{9}
(Kumchev).Let
\begin{equation}
S_{k}(N,\alpha)=\sum\limits_{p\sim N}(\log p)e(p^{k}\alpha)\notag
\end{equation}
where $\alpha=a/q+\lambda$ subject to $1\leq a\leq q\leq N$, $(a,q)=1$ and $|\lambda|\leq 1/qQ$, with
\begin{equation}
Q=\left \{ \begin{array}{lll}
   x^{3/2},\qquad &if\ k=2,\\
   x^{12/7},\quad &if\ k=3.
 \end{array} \right.\notag
\end{equation}
Then
\begin{equation}
\begin{aligned}
S_{k}(N,\alpha)&\ll N^{1-\varrho+\varepsilon}+\frac{q^{\varepsilon}NL^{c}}{\sqrt{q(1+|\lambda|N^{k})}}, \\
&\varrho = \left \{ \begin{array}{lll}
   {1/8},\qquad &if\ k=2,\\
   1/14,\quad &if\ k=3.
 \end{array} \right.\notag
\end{aligned}
\end{equation}
\end{lemma}

The two results on exponential sums over primes can be found in \cite{1997v} and \cite{2006k}. On the minor arcs, we need estimates for the measure of the set $\sigma_{\lambda}(\alpha)$ defined in (\ref{10}).

\begin{lemma}\label{11}
We have
\begin{equation}
meas(\sigma_{\lambda}(\alpha))\ll N_{i}^{-E(\lambda)},\quad E(0.961917)>\frac{113}{126}+10^{-10}.\notag
\end{equation}
\end{lemma}

\begin{proof}
The definition of $E(\lambda)$ was constructed by Heath-Brown and Puchta in \cite{2002h} and we take $\lambda=0.961917$.
\end{proof}

A crucial step is the next two lemmas which give the estimates of two integrations used in estimating $R_{3}(N_{1},N_{2})$. They are obtained by Ren \cite{2001r}\cite{2003r} and Kong \cite{2017k}.

\begin{lemma}\label{12}
Let $n$ be an even integer, and $\rho(n)$ the number of representations of n in the form
\begin{equation}
n=p_{1}^{3}+\cdots+p_{4}^{3}-p_{5}^{3}-\cdots-p_{8}^{3},\quad 0\leq |n|\leq N_{i},\notag
\end{equation}
and subject to
\begin{equation}
p_{1},p_{2},p_{5},p_{6}\sim U_{i}\quad p_{3},p_{4},p_{7},p_{8}\sim V_{i}.\notag
\end{equation}
Then we have
\begin{equation}
\rho(n)\leq bU_{i}V_{i}^{4}(\log N_{i})^{-8}, \notag
\end{equation}
with $b=268096$.
\end{lemma}

\begin{lemma}\label{13}
Let
\begin{equation}
J=\sum\limits_{1\leq m_{1},\ldots m_{4}\leq L}\prod\limits_{i=1}^{2}r_{i}(2^{m_{1}}+2^{m_{2}}-2^{m_{3}}-2^{m_{4}}),\notag
\end{equation}
where
\begin{equation}
r_{i}(n)=\#\{\omega N<p_{i}\leq N_{i}:n=p_{1}-p_{2}\}.\notag
\end{equation}
with $\omega$ is a small positive constant. Then we have
\begin{equation}
J\leq 305.8869\frac{N_{1}N_{2}L^{4}}{(\log N_{1}\log N_{2})^{2}}.\notag
\end{equation}
\end{lemma}

\section{\bf Proof of Theorem 1.1}

We begin with estimating $R_{1}(N_{1},N_{2})$.

\begin{lemma}\label{3.1}
\begin{equation}
R_{1}(N_{1},N_{2})\geq 0.00089051(1-\varepsilon)N_{1}^{\frac{11}{9}}N_{2}^{\frac{11}{9}}L^{k}+O\left(N_{1}^{\frac{11}{9}}N_{2}^{\frac{11}{9}}L^{k-1}\right).
\end{equation}
\end{lemma}

\begin{proof}
Subject to the definition of $\Xi(N,k)$, every
\begin{equation}
n_{i}=N_{i}-2^{v_{1}}-\cdots -2^{v_{k}}\notag
\end{equation}
with $v_{i}\leq L$ will be concluded in $\Xi(N_{i},k)$. We apply lemma \ref{6} and \ref{7} to get that
\begin{equation}
\begin{aligned}
R_{1}(N_{1},N_{2})&=\frac{1}{3^{8}}\sum\limits_{n_{1}\in \Xi(N_{1},k)\atop n_{2}\in \Xi(N_{2},k)}\left (\mathfrak{S}(n_{1})\mathfrak{J}_{i}(n_{1})+O(N_{1}^{\frac{11}{9}}L^{-1})\right)\left(\mathfrak{S}(n_{2})\mathfrak{J}_{i}(n_{2})+O(N_{2}^{\frac{11}{9}}L^{-1})\right)\\
& \geq 0.00089051\sum\limits_{n_{1}\in \Xi(N_{1},k)\atop n_{2}\in \Xi(N_{2},k)}N_{1}^{\frac{11}{9}}N_{2}^{\frac{11}{9}}+O\left(N_{1}^{\frac{11}{9}}N_{2}^{\frac{11}{9}}L^{k-1}\right)\\
& \geq 0.00089051(1-\varepsilon)N_{1}^{\frac{11}{9}}N_{2}^{\frac{11}{9}}L^{k}+O\left(N_{1}^{\frac{11}{9}}N_{2}^{\frac{11}{9}}L^{k-1}\right).\notag
\end{aligned}
\end{equation}
\end{proof}

For $R_{2}(N_{1},N_{2})$, we will use the untrivial estimates for $f_{i}(\alpha)$ and $S_{i}(\alpha)$ and the measure of $\sigma(\lambda)$. Applying lemma \ref{8} and \ref{9} we have
\begin{equation}\label{15}
\begin{aligned}
&\max\limits_{\alpha\in\mathfrak{m}_{i}}|f_{i}(\alpha)|\ll N_{i}^{1-1/18+\varepsilon},\\
&\max\limits_{\alpha\in\mathfrak{m}_{i}}|S_{i}(\alpha)|\ll N_{i}^{1/3-1/42+\varepsilon}.
\end{aligned}
\end{equation}

\begin{lemma}\label{3.2}
\begin{equation}
R_{2}(N_{1},N_{2})\ll N_{1}^{\frac{11}{9}}N_{2}^{\frac{11}{9}}L^{k-1}.\nonumber
\end{equation}
\end{lemma}

\begin{proof}
By the definition of $\mathfrak{m}$, we have
\begin{equation}
\mathfrak{m}\subset \{(\alpha_{1},\alpha_{2}):\alpha_{1}\in \mathfrak{m}_{1},\alpha_{2}\in [0.1]\}\bigcup \{(\alpha_{1},\alpha_{2}):\alpha_{1}\in [0,1],\alpha_{2}\in \mathfrak{m}_{2}\}\notag
\end{equation}
Then we have
\begin{align}\label{14}
R_{2}(N_{1},N_{2})&\leq L^{k}\left ( \iint\limits_{(\alpha_{1},\alpha_{2})\in\mathfrak{m}_{1}\times[0,1]\atop |G(\alpha_{1}+\alpha_{2})|\geq \lambda L}+\iint\limits_{(\alpha_{1},\alpha_{2})\in[0,1]\times\mathfrak{m}_{2}\atop |G(\alpha_{1}+\alpha_{2})|\geq \lambda L}\right ) \prod_{i=1}^{2}f_{i}(\alpha_{i})S_{i}^{2}(\alpha_{i})T_{i}^{2}(\alpha_{i})e(-\alpha_{i}N_{i})d\alpha_{1}d\alpha_{2} \nonumber\\
&=L^{k}\left(\iint\limits_{1}+\iint\limits_{2}\right),
\end{align}
where we use the trivial bound of $G(\alpha_{1}+\alpha_{2})$. Due to symmetry of $\alpha_{1}$ and $\alpha_{2}$ in the integration, we only estimate one integration on the right of (\ref{14}). We use lemma \ref{11}, lemma \ref{12}, (\ref{15}), the trivial bound of $T_{i}(\alpha)$ and Cauchy-Schwarz inequality to get
\begin{align}\label{16}
\iint\limits_{1}&\ll N_{1}^{1-1/18+2/3-2/42+5/9+\varepsilon}\int_{0}^{1}|f_{2}(\alpha_{2})S_{2}^{2}(\alpha_{2})T_{2}^{2}(\alpha_{2})|\left(\int\limits_{\beta\in[\alpha_{2},1+\alpha_{2}]\atop G(\beta)\geq \lambda L}d\beta\right)d\alpha_{2}\nonumber\\
&\ll N_{1}^{1+11/9-13/126+\varepsilon}N_{1}^{-E(\lambda)}\left(\int_{0}^{1}|f_{2}(\alpha_{2})|^{2}d\alpha_{2}\right)^{1/2}\left(\int_{0}^{1}|S_{2}(\alpha_{2})|^{4}|T_{2}(\alpha_{2})|^{4}d\alpha_{2}\right)^{1/2}\nonumber\\
&\ll N_{1}^{\frac{11}{9}}N_{2}^{\frac{11}{9}}L^{-1},\nonumber
\end{align}
where set $\beta=\alpha_{1}+\alpha_{2}$ to give the integral transformation and we used the periodicity of $G(\alpha)$ and the condition $N_{2}\gg N_{1}\geq N_{2}$. Similarly,
\begin{align}
\iint\limits_{2}&\ll N_{1}^{\frac{11}{9}}N_{2}^{\frac{11}{9}}L^{-1}.\nonumber
\end{align}
This yields
\begin{align}
R_{2}(N_{1},N_{2})\ll N_{1}^{\frac{11}{9}}N_{2}^{\frac{11}{9}}L^{k-1}.\nonumber
\end{align}
\end{proof}

Finally, we treat with the $R_{3}(N_{1},N_{2})$. By using lemma \ref{12} and (\ref{2.1}) we can get the following lemma.
\begin{lemma}\label{3.31}
Let $S_{i}(\alpha)$ and $T_{i}(\alpha)$ be as in (\ref{10}). Then
\begin{equation}
\int_{0}^{1}|S_{i}(\alpha)T_{i}(\alpha)|^{4}d\alpha\leq 0.359127N_{i}^{\frac{13}{9}}.\notag
\end{equation}
\end{lemma}

\begin{lemma}\label{3.3}
\begin{equation}
R_{3}(N_{1},N_{2})\ll 6.2809957\lambda^{k-3}N_{1}^{\frac{11}{9}}N_{2}^{\frac{11}{9}}L^{k}.\nonumber
\end{equation}
\end{lemma}

\begin{proof}
By the definition of $R_{3}(N_{1},N_{2})$ and Cauchy-Schwarz inequality, we have
\begin{align}\label{l3}
R_{3}(N_{1},N_{2})&\leq (\lambda L)^{k-2}\int_{0}^{1}\int_{0}^{1}\left|\left(\prod\limits_{i=1}^{2}f_{i}(\alpha_{i})S_{i}^{2}(\alpha_{i})T_{i}^{2}(\alpha_{i})\right)G^{2}(\alpha_{1}+\alpha_{2})\right|d\alpha_{1}d\alpha_{2}\nonumber \\
&\leq (\lambda L)^{k-2}\left(\iint\limits_{[0,1]^{2}}\left|f_{1}^{2}f_{2}^{2}G^{4}\right|d\alpha_{1}d
\alpha_{2}\right)^{1/2}\left(\iint\limits_{[0,1]^{2}}\left|S_{1}^{4}T_{1}^{4}S_{2}^{4}T_{2}^{4}\right|d\alpha_{1}d
\alpha_{2}\right)^{1/2}\nonumber\\
&= (\lambda L)^{k-2}J^{1/2}I^{1/2}.
\end{align}
Observe that
\begin{align}
J&=\iint\limits_{[0,1]^{2}}\left|f_{1}^{2}(\alpha_{1})f_{2}^{2}(\alpha_{2})G^{4}(\alpha_{1}+\alpha_{2})\right|d\alpha_{1}d\alpha_{2}\nonumber\\
&=\sum\limits_{1\leq m_{1},\ldots m_{4}\leq L}\prod\limits_{i=1}^{2}t_{i}(2^{m_{1}}+2^{m_{2}}-2^{m_{3}}-2^{m_{4}}),\nonumber
\end{align}
where
\begin{align}
t_{i}(n)=\sum\limits_{\omega N_{i}<p_{1},p_{2}\leq N_{i}\atop p_{1}-p_{2}=n}\log p_{1}\log p_{2}.\notag
\end{align}
Applying lemma \ref{13} we get
\begin{align}\label{l1}
J&\leq (\log N_{1}\log N_{2})^{2}\sum\limits_{1\leq m_{1},\ldots m_{4}\leq L}\prod\limits_{i=1}^{2}r_{i}(2^{m_{1}}+2^{m_{2}}-2^{m_{3}}-2^{m_{4}})\notag\\
&\leq 305.8869N_{1}N_{2}L^{4}.
\end{align}
For $I$ we use lemma \ref{3.31} to get
\begin{align}\label{l2}
I&\leq 0.359127^{2}N_{1}^{\frac{13}{9}}N_{2}^{\frac{13}{9}}.
\end{align}
Taking (\ref{l1}) and (\ref{l2}) in (\ref{l3}), we can obtain the conclusion.
\end{proof}

Combining lemma \ref{3.3} with lemma \ref{3.2} and lemma \ref{3.1}, we get
\begin{align}
R(N_{1},N_{2})\geq (0.00089051-6.2809957\lambda^{k-2})(1-\varepsilon)N_{1}^{\frac{11}{9}}N_{2}^{\frac{11}{9}}L^{k}.\notag
\end{align}
Recall that $\lambda=0.961917$. When $k\geq 231$ and $\varepsilon=10^{-10}$, we have
\begin{align}\label{l2}
R(N_{1},N_{2})> 0.\notag
\end{align}
for sufficiently large integers $N_{1}$ and $N_{2}$ satisfying the given condition.

\bibliographystyle{plain}
\bibliography{ck-bib}

 \end{document}